\documentclass[10pt, reqno, a4paper]{amsart}

\usepackage[latin1]{inputenc}
\usepackage{amsfonts}
\usepackage{amsmath}
\usepackage{amssymb}
\usepackage{amsthm}
\usepackage[T1]{fontenc}
\usepackage{enumerate}
\usepackage{verbatim}
\usepackage{lmodern}

\usepackage[pdftex]{hyperref}
\usepackage[matrix, arrow]{xy}
\usepackage{listings}
\usepackage[french, english]{babel}

\title{Sur une question de Mehta et Pauly}
\author{Holger Brenner}
\author{Axel St\"abler}
\thanks{The second author was supported by SFB/TRR 45 Bonn-Essen-Mainz financed by Deutsche Forschungsgemeinschaft}

\DeclareMathOperator{\Ext}{Ext}

\input xy
\xyoption{all}
\SelectTips{cm}{10}
\makeatletter
\newenvironment{abstracts}{%
  \ifx\maketitle\relax
    \ClassWarning{\@classname}{Abstract should precede
      \protect\maketitle\space in AMS document classes; reported}%
  \fi
  \global\setbox\abstractbox=\vtop \bgroup
    \normalfont\Small
    \list{}{\labelwidth\z@
      \leftmargin3pc \rightmargin\leftmargin
      \listparindent\normalparindent \itemindent\z@
      \parsep\z@ \@plus\p@
      
      \itemsep\medskipamount
    }%
}{%
  \endlist\egroup
  \ifx\@setabstract\relax \@setabstracta \fi
}

\newcommand{\abstractin}[1]{%
  \otherlanguage{#1}%
  \item[\hskip\labelsep\scshape\abstractname.]%
}
\makeatother

\begin{document}
\swapnumbers
\theoremstyle{definition}
\newtheorem{Le}{Lemma}[section]
\newtheorem{Def}[Le]{Definition}
\newtheorem*{DefB}{Definition}
\newtheorem{Bem}[Le]{Remark}
\newtheorem{Ko}[Le]{Corollary}
\newtheorem{Theo}[Le]{Theorem}
\newtheorem*{TheoB}{Theorem}
\newtheorem{Bsp}[Le]{Example}
\newtheorem{Be}[Le]{Observation}
\newtheorem{Prop}[Le]{Proposition}
\newtheorem*{PropB}{Proposition}
\newtheorem{Sit}[Le]{Situation}
\newtheorem{Que}[Le]{Question}
\newtheorem{Con}[Le]{Conjecture}
\newtheorem{Dis}[Le]{Discussion}
\newtheorem{Prob}[Le]{Problem}
\newtheorem*{Konv}{Convention}
\def\cocoa{{\hbox{\rm C\kern-.13em o\kern-.07em C\kern-.13em o\kern-.15em
A}}}
\address{Holger Brenner\\ Universit\"at Osnabr\"uck\\ Fachbereich 6\\Albrechtstr.\ 28a\\ 49069 Osnabr\"uck\\ 
Germany}
\email{hbrenner@uni-osnabrueck.de}
\address{Axel St\"abler\\
Johannes Gutenberg-Universit\"at Mainz\\ Fachbereich 8\\
Staudingerweg 9\\
55099 Mainz\\
Germany}

\email{staebler@uni-mainz.de}

\date{\today}
\subjclass[2010]{Primary 14H60}
\begin{abstracts}
\abstractin{french}
Dans cette courte note, nous donnons des exemples explicites en characteristique $p$ sur certaines courbes projectives lisses o\`u, pour un fibr\'{e} vectoriel semi-stable donn\'{e} $\mathcal{E}$, la longeur de la filtration d'Harder-Narasimhan de $F^\ast \mathcal{E}$ est plus grande que $p$. Cela r\'epond negativement \`a une question pos\'ee par Mehta et Pauly dans \cite{mehtapaulysemistabledirectimage}.

\abstractin{english}
In this short note we provide explicit examples in characteristic $p$ on certain smooth projective curves where for a given semistable vector bundle $\mathcal{E}$ the length of the Harder-Narasimhan filtration of $F^\ast \mathcal{E}$ is longer than $p$. This negatively answers a question of Mehta and Pauly raised in \cite{mehtapaulysemistabledirectimage}.
\end{abstracts}

\maketitle
\section*{Introduction}
\selectlanguage{english}

In \cite[page 2]{mehtapaulysemistabledirectimage} Mehta and Pauly asked whether for a smooth projective curve over a field of characteristic $p > 0$ and $\mathcal{E}$ a semistable bundle on $X$ the length of the Harder-Narasimhan filtration of $F^\ast \mathcal{E}$ is at most $p$. In \cite[Construction 2.13]{zhoucounterexampleshnflength} this is answered negatively. Examples are constructed based on a result of Sun (\cite{sunfrobdirectimage}). The bundles for which examples are obtained in \cite{zhoucounterexampleshnflength} have rank $\geq 2p$ (in fact, examples are constructed for any $np$ with $ n \geq 2$) and are over curves of large genus since restriction theorems and Bertini are used. The purpose of this short note is to provide surprisingly simple down to earth examples in characteristic $p$ for certain smooth plane curves and bundles of rank $p +1 \leq r \leq \lfloor \frac{3p+1}{2} \rfloor$. In characteristic $2$ negative examples exist on any smooth projective curve of genus $\geq 2$. We note that our examples are only polystable, while one should be able to obtain stable bundles using the methods outlined in \cite{zhoucounterexampleshnflength}.

\section{The example}

\begin{Prop}
\label{HNF}
Let $X$ be a smooth projective curve over an algebraically closed field $k$ of positive characteristic. Let $\mathcal{E}_i$, $i=1, \ldots, n$ be semistable rank two bundles of slope $\mu$ on $X$ such that the $F^\ast \mathcal{E}_i$ split as $F^\ast \mathcal{E}_i = \mathcal{L}_i \oplus \mathcal{G}_i$ with $\mu(\mathcal{L}_i) > \mu(\mathcal{G}_i)$. Assume moreover, that $\mu(\mathcal{L}_i) > \mu(\mathcal{L}_{i+1})$ for all $i = 1, \ldots, n-1$. Then $\mathcal{S} = \bigoplus_{i=1}^n \mathcal{E}_i $ is semistable and $F^\ast \mathcal{S}$ is unstable and its Harder-Narasimhan filtration is \[0 \subset \mathcal{L}_1 \subset \mathcal{L}_1 \oplus \mathcal{L}_2 \subset \ldots \subset \bigoplus_{i=1}^n \mathcal{L}_i \subset \bigoplus_{i=1}^n \mathcal{L}_i  \oplus \mathcal{G}_n \subset \bigoplus_{i=1}^n \mathcal{L}_i  \oplus \mathcal{G}_n \oplus \mathcal{G}_{n-1} \subset \ldots \subset  F^\ast \mathcal{S}.\] In particular, the Harder-Narasimhan filtration of $F^\ast \mathcal{S}$ has length $2n$.
\end{Prop}
\begin{proof}
Clearly $\mathcal{S}$ is semistable. We have $\mu(\mathcal{G}_i) = 2\mu - \mu(\mathcal{L}_i)$ which implies $\mu(\mathcal{G}_i) < \mu(\mathcal{G}_{i+1})$ for all $i$. We also have $\mu(\mathcal{L}_i) > \mu(\mathcal{G}_j)$ for all $i,j$. Indeed, we may assume that $i > j$ then $\mu(\mathcal{L}_i) - \mu(\mathcal{G}_j) = \mu(\mathcal{L}_j) - \mu(\mathcal{G}_i)$ and by assumption $\mu(\mathcal{L}_i) > \mu(\mathcal{L}_j) > \mu(\mathcal{G}_j)$. Hence, $\mu(\mathcal{L}_j) > \mu(\mathcal{G}_i)$.

It follows that the slopes of the quotients $\mathcal{Q}_i$ of the filtration form a strictly decreasing sequence. As all $\mathcal{Q}_i$ are semistable as line bundles this is the Harder-Narasimhan filtration of $F^\ast \mathcal{S}$.
\end{proof}

\begin{Bsp}
By \cite[Theorem 1]{langepaulyfrobenius} any smooth projective curve $X$ of genus $\geq 2$ admits a semistable rank two bundle $\mathcal{E}$ with trivial determinant such that $F^\ast \mathcal{E}$ is not semistable. Then $\mathcal{S} = \mathcal{E} \oplus \mathcal{O}_X$ is a semistable vector bundle and the Harder-Narasimhan filtration of $F^\ast \mathcal{S}$ has length $3 > 2$. Indeed, if $0 \subset \mathcal{L} \subset F^\ast \mathcal{E}$ is a Harder-Narasimhan filtration of $F^\ast \mathcal{E}$ then $0 \subset \mathcal{L} \subset \mathcal{L} \oplus \mathcal{O}_X \subset F^\ast \mathcal{S}$ is one for $F^\ast \mathcal{S}$.
\end{Bsp}

\begin{Le}
\label{Lemma}
Let $X$ be a smooth projective curve and $\mathcal{E}$ a rank $2$ vector bundle on $X$. If $\mathcal{E}$ is given by an extension $0 \neq c \in \Ext^1(\mathcal{M}, \mathcal{L})$ with $\deg \mathcal{L} < \deg \mathcal{M}$ and $F^\ast (c) = 0$ then $\mathcal{E}$ is semistable.
\end{Le}
\begin{proof}
Assume to the contrary that $\mathcal{E}$ is unstable and let $\mathcal{N}$ denote the maximal destablizing subbundle $\mathcal{E}$. The maximal destabilising subbundle of $F^\ast \mathcal{E}= F^\ast \mathcal{M} \oplus F^\ast \mathcal{L}$ is $F^\ast \mathcal{M}$.
Since the Harder-Narasimhan filtration is unique and in the rank $2$ case automatically strong we must have $F^\ast \mathcal{M} = F^\ast \mathcal{N}$. Hence, $\mathcal{N} = \mathcal{M} \otimes \mathcal{T}$ for some $p$-torsion bundle $\mathcal{T}$.

Consider now the natural inclusion $i: \mathcal{M} \otimes \mathcal{T} \to \mathcal{E}$ and the projection $p: \mathcal{E} \to \mathcal{M}$. The Frobenius pull-back of the composition $p \circ i$ is the identity. In particular $p \circ i: \mathcal{M} \otimes \mathcal{T} \to \mathcal{M}$ is non-zero. Since both line bundles are of the same degree this map is an isomorphism. Hence, if $\mathcal{E}$ is not semistable then the sequence has to split which contradicts the assumption $c \neq 0$.
\end{proof}

\begin{Bsp}
Let now $p$ be any prime and $k$ an algebraically closed field of characteristic $p$. We consider the plane curve \[X = V_+(x^{3p} + xy^{3p-1} + yz^{3p-1}) \subseteq \mathbb{P}^2_k.\] By the Jacobian criterion this is a smooth curve. We will construct $\lfloor \frac{3p+1}{2} \rfloor$ rank two bundles of slopes $-\frac{3p}{2}$ as in Proposition \ref{HNF}. The direct sum over at least $\frac{p+1}{2}$ of these bundles then constitutes the desired example. 

Consider the cohomology class \[c = \frac{x^3}{y^2z^2} \in H^1(X, \mathcal{O}_X(-1))\] which is non-zero. Also note that its Frobenius pull-back \[F^\ast(c) = \frac{x^{3p}}{y^{2p}z^{2p}} = \frac{-xy^{3p-1} - yz^{3p-1}}{y^{2p}z^{2p}}= -(\frac{xy^{p-1}}{z^{2p}} + \frac{z^{p-1}}{y^{2p-1}})\] is zero. Moreover, multiplication by $z$ yields a map $\mathcal{O}_X(-1) \to \mathcal{O}_X$ and the induced map on cohomology maps $c$ to $\frac{x^4}{y^2z^2}$ which is still non-zero. Let $P_1, \ldots, P_{3p}$ be the (distinct) points on $X$ where $z$ vanishes\footnote{We could also work with multiplication by $x$ which yields one reduced point and one with multiplicity $3p -1$.}. In particular, the cokernel of multiplication by $z$ is just $\bigoplus_{i=1}^{3p} k(P_i)$, where $k(P_i)$ is the skyscraper sheaf at $P_i$.

Multiplication by $z$ factors as \[\begin{xy}\xymatrix{\mathcal{O}_X(-1) \ar[r]& \mathcal{O}_X(-1 + \sum_{i=1}^l P_i) \ar[r]& \mathcal{O}_X}\end{xy}\] for any $l \leq 3p$. Indeed, the image of the line bundle in the middle is just the sum of the image of $\mathcal{O}_X(-1)$ in $\mathcal{O}_X$ and the preimage of $\sum_{i=1}^l k(P_i)$. 
 In particular, we get an induced factorization on cohomology and we denote the image of $c$ in $H^1(X,\mathcal{O}_X(-1 + \sum_{i=1}^l P_i))$ by $c_l$. Note that $c_l$ is non-zero, while $F^\ast(c_l)$ is zero.

Assume now that $l$ is even. These cohomology classes then define extensions $\mathcal{E}_l$ as follows. Let $I$ be the odd numbers from $1$ to $l$ and let $J$ be the even numbers from $1$ to $l$. Then the \[c_l \in H^1(X, \mathcal{O}_X(-1 + \sum_{i=1}^l k(P_i))) = \Ext^1(\mathcal{O}_X(-\sum_{j \in J} P_j), \mathcal{O}_X(-1 + \sum_{i \in I} P_i))\] yield  extensions \[\begin{xy}\xymatrix{0 \ar[r]& \mathcal{O}_X(-1 + \sum_{i \in I} P_i) \ar[r]& \mathcal{E}_l \ar[r] & \mathcal{O}_X(-\sum_{j \in J} P_j) \ar[r]& 0} \end{xy}.\]

The $\mathcal{E}_l$ all have slope $-\frac{3p}{2}$ and pulling back along Frobenius splits the above sequence. By Lemma \ref{Lemma} the $\mathcal{E}_l$ are semistable. Hence, the $\mathcal{E}_l$ satisfy the hypothesis of Proposition \ref{HNF} and we obtain the desired examples.
\end{Bsp}

\bibliography{bibliothek.bib}

\providecommand{\bysame}{\leavevmode\hbox to3em{\hrulefill}\thinspace}
\providecommand{\MR}{\relax\ifhmode\unskip\space\fi MR }
\providecommand{\MRhref}[2]{%
  \href{http://www.ams.org/mathscinet-getitem?mr=#1}{#2}
}
\providecommand{\href}[2]{#2}
\begin{thebibliography}{1}

\bibitem{langepaulyfrobenius}
H.~Lange and C.~Pauly, \emph{On {F}robenius-destabilized rank-2 vector bundles
  over curves}, Comm. Math. Helv. \textbf{83} (2008), 179--209.

\bibitem{mehtapaulysemistabledirectimage}
V.~B. Mehta and C.~Pauly, \emph{Semistability of {F}robenius direct images over
  curves}, Bull. Soc. Math. France \textbf{135} (2007), no.~1, 105--117.

\bibitem{sunfrobdirectimage}
X.~Sun, \emph{Direct images of bundles under {F}robenius morphism}, Invent.
  Math. 173 \textbf{173} (2008), 427--447.

\bibitem{zhoucounterexampleshnflength}
M.~Zhou, \emph{The {H}-{N} filtration of bundles as {F}robenius pull-back},
  preprint, arXiv:1212.4404v1 (2012).

\end{thebibliography}
\bibliographystyle{amsplain}
\end{document}